\documentclass[12pt]{amsart}
\usepackage{xcolor}
\usepackage[normalem]{ulem}
\usepackage[T1]{fontenc}

\usepackage[margin=3cm]{geometry}
\usepackage{amsthm, bbm}
\usepackage{amsmath}
\usepackage{hyperref}
\usepackage{todonotes}
\usepackage{dsfont}

\theoremstyle{definition}
\newtheorem{definition}{Definition}
\newtheorem{theorem}[definition]{Theorem}
\newtheorem{lemma}[definition]{Lemma}
\newtheorem{corollary}[definition]{Corollary}
\newtheorem{proposition}[definition]{Proposition}
\newtheorem{remark}[definition]{Remark}

\def\cG{{\mathcal{G}}}
\def\cX{{\mathcal{X}}}
\def\DS{\displaystyle}
\def\eps{{\varepsilon}}
\def\id{{\mathrm{id}}}

\def\Prob{{\mathbb{P}}}
\def\EXP{{\mathbb{E}}}
\def\fS{{\mathbb{S}}}
\def\T{\mathbb{T}}

\def\tq{{\widetilde{q}}}
\def\tX{{\tilde X}}

\def\p{\mathfrak{p}}
\def\q{\mathfrak{q}}
\def\mr{\eta}
\def\wT{\widetilde{T}}

\def\Var{{\mathrm{Var}}}

\def\hphi{{\hat\varphi}}

\def\Leb{{\mathrm{Leb}}}

\title[The central limit theorem and rate of mixing...]
{The central limit theorem and rate of mixing for simple random walks on the circle}
\author{Klaudiusz Czudek}
\address{Klaudiusz Czudek, Institute of Science and Technology Austria (ISTA), Am Campus 1, 3400 Klosterneuburg}
\email{klaudiusz.czudek@gmail.com}

\author{Dmitry Dolgopyat}
\address{Dmitry Dolgopyat, University of Maryland,  Institute for Physical Science and Technology,
Atlantic Building College Park, MD 20742}
\email{dimadolgopyat@gmail.com}



\begin{document}

\begin{abstract}
We prove the Central Limit Theorem and superpolynomial mixing for environment viewed for the particle process in 
quasi periodic Diophantine random environment. The main ingredients are smoothness estimates for the solution of the Poisson 
equation and local limit asymptotics for certain accelerated walks.
\end{abstract}

\maketitle

\section{Introduction}

Environment viewed by the particle (EVP) 
process is a powerful tool in the study of random walks 
in random environment, (see \cite{Kozlov, BS02, Zeitouni04} and the references wherein).
In particular, for random walks on $\mathbb{Z}^d$, 
if this process possesses a stationary measure which is absolutely continuous 
with respect to the environment measure, and if the stationary measure has good mixing properties, then the corresponding random walk in random environment satisfies the Central 
Limit Theorem. However, good mixing properties of EVP are far from given.
In fact, 
the EVP process is a standard source of examples of exotic
behavior in both ergodic theory \cite{Kal82, dHS97} and limit theory \cite{KS79, Bolt89}. 
\\

The EVP process could be defined for random walks on arbitrary groups.
Consider a semigroup $G$ with a finite set of generators $\Gamma\ni\mathrm{id}$ 
acting on a compact metric space $\cX$. For every $\gamma\in\Gamma$ fix a positive continuous function $p(\cdot, \gamma):\cX\rightarrow [0,1]$, $\gamma \in \Gamma$, in such a way that $\DS \sum_{\gamma\in \Gamma} p(x, \gamma)\!\!=\!\!1$ for every $x\in \cX$. The formula
\begin{equation}
\label{E:transfer_operator}
T\phi(x) = \sum_{\gamma\in \Gamma} p(x, \gamma) \phi(\gamma x ),
\end{equation}
defines a Markov operator $T$ on the space $C(\cX)$ of continuous functions on $\cX$ that gives rise to a Markov process $(X_n)$ on $\cX$ which we call 
{\em environment viewed by the 
particle for RWRE on $G$. } We note that the operator \eqref{E:transfer_operator}
also appears naturally 
 in the theory of iterated function systems and its properties are widely studied in mathematics.
 
 If $\mu$ is a Borel probability measure on $\cX$ then by $\Prob_\mu$ (resp. $\EXP_\mu$) we denote the conditional probability (resp. expectation) given the distribution of $X_0$ is $\mu$. A Borel measure $\mu$ is called stationary if $\mu(T\phi)=\mu(\phi)$ for every $\phi \in C(\cX)$, where $\mu(\phi)$ stands for $\int_{\cX} \phi d\mu$. A stationary measure $\mu$ is mixing if $\EXP_\vartheta \phi(X_n) \to \mu(\phi)$ for any $\phi \in C(\cX)$ and any measure $\vartheta$ absolutely continuous with respect to $\mu$.\\

The application of the Krylov-Bogoliubov procedure and the compactness of $\cX$ yield the existence of a stationary distribution for 
the operator \eqref{E:transfer_operator},
however the uniqueness and mixing properties are not well understood.
 We refer to \cite{Conze_Guivarch_00, BS02, Bremont02, Bremont09} for discussions of
 the stationary measures
 for the case of $\mathbb{Z}^d$, to 
 \cite{DG19, Dolgopyat_Goldsheid_21} for case of finite extensions of $\mathbb{Z}$,
  and to
 \cite{BFLM, BQIII} for the case of larger groups, where even the case of constant 
 transition probabilities is far from settled.
 The mixing of the EVP process
in the case of the independent environment on $\mathbb{Z}$ was studied in 
\cite{Kesten77, Lalley86}, and a simpler proof could be found in \cite{DG13}. 

We note that there are several results on uniqueness and mixing under the assumption that \eqref{E:transfer_operator} is contracting in average or some related conditions (e.g. \cite{Barnsley_Demko_Elton_Geronimo_1988, Czudek_20, Doeblin_Fortet_1937, Onicescu_Mihoc_1935, Sleczka_11}, see also a survey \cite{Stenflo_12}). The contraction condition typically does not hold in the case of random 
walk in random environment, since in that case one typically assumes that the $G$ action
on $\cX$ preserves some non atomic probability measure.
\\

 Our goal in this paper is to go one step
beyond existence of stationary measures, and to study their ergodic properties.
Our results can be divided into two parts. First, we show that in the quite general setting unique
ergodicity of stationary measure implies mixing. In the second part we employ our approach
in the simplest possible setting: quasi periodic random walks on $\mathbb{Z}$
and work out precise mixing estimates in that case using the harmonic analysis on
the circle. As a direct application of our result we obtain the central limit theorem 
for a functional of the environment observed by the particles in that case.
We hope that our approach could be useful both for studying walks on more 
complicated groups and for the studies of further ergodic properties in the quasiperiodic 
environment. \\

Our first result says that if $\cX$ is compact and the random walk 
is lazy (recall that a walk is  lazy if it can stay at the same place for several units of time).
then the unique ergodicity implies mixing.

\begin{theorem}\label{theorem_mixing}
Suppose that the random walk $(X_n)$ with the Markov operator \eqref{E:transfer_operator} is uniquely ergodic (i.e. it has exactly one stationary measure $\nu$), $\cX$ is compact, $\id\in \Gamma$, the functions $p(\cdot, \gamma)$ are continuous for $\gamma\in\Gamma$ and there exists $\eps_0$ such that \\$\eps_0 \leq p(x, \id) \leq 1-\eps_0$ for every $x\in \cX$. Then
\begin{equation}
\label{E:mixing}
\lim_{n\to\infty} \EXP_x(\phi(X_n))=\nu(\phi) \quad \textrm{for every $x\in\cX$ and $\phi\in C(\cX)$.}
\end{equation}
\end{theorem}

\noindent We give two applications of this theorem.

\begin{corollary}
Let $f_1, \ldots, f_N$ be circle homeomorphisms, $p_1, \ldots p_N$ be positive numbers with $p_1+\cdots+p_N=1$. Let $(X_n)$ be a Markov process with the Markov operator
$$T\phi(x)= \sum_{i=1}^N p_i \phi(f_i(x)), \quad \phi \in C(\T).$$
If one of $f_i$'s is the identity and the action on the circle is minimal, then \eqref{E:mixing} holds.
\end{corollary}
\begin{proof}
The unique ergodicity follows from 
\cite[Corollary 2.3]{Malicet_17}.
\end{proof}

To give another application, fix an irrational number $\alpha$, a positive continuous function $\p\in C(\T)$ with $0<\p(x)<1$, $x\in\T$, and define
\begin{equation}\label{walk_def}
T\psi(x) = \p(x) \psi(x+\alpha)+ \q(x) \psi(x-\alpha), \quad T: C(\T)\rightarrow C(\T),
\end{equation}
where $\q(x)=1-\p(x)$.

This random walk was considered by Sinai \cite{Sinai_99}, where the author proved the unique ergodicity and mixing under assumptions that $\alpha$ is Diophantine and $\p$ is sufficiently smooth. Later Kaloshin and Sinai \cite{Kaloshin_Sinai_00} showed there is no need to assume that $\alpha$ is Diophantine\footnote{A number $\alpha\not\in \mathbb{Q}$ is called Diophantine of type $(c, \tau)$, $c>0$, $\tau\ge 0$, if
$$\bigg| \alpha - \frac{p}{q} \bigg| \ge \frac{c}{q^{2+\tau}} \quad \textrm{for every $p, q\in\mathbb{Z}, q\not = 0$.}$$
A number $\alpha$ is called Liouville when it is not Diophantine of any type.} when $\p$ is asymmetric, i.e. when
$$
\int_{\T} \log\frac{\p(x)}{\q(x)} dx \not = 0.
$$
The unique ergodicity in the symmetric case (i.e. when the integral above is zero) for any irrational $\alpha$ has been proven by Conze and Guivarc'h \cite{Conze_Guivarch_00}. Another proof of this result has been given in \cite{Czudek_24}. It has been proven there also that for any $\alpha$ irrational \eqref{walk_def} is mixing for a generic choice of $\p$. With Theorem \ref{theorem_mixing} we can strengthen the latter result.

\begin{corollary}
Let $\log \frac{\p(x)}{\q(x)}$ be continuous of bounded variation, $\alpha\not\in\mathbb{Q}$. Then  \eqref{E:mixing} holds.
\end{corollary}
\begin{proof}
 The process $(X_{n})$ restricted to even steps  satisfies the assumptions of Theorem \ref{theorem_mixing} by \cite{Conze_Guivarch_00}, thus \eqref{E:mixing} holds for $(X_{2n})$. Then
conditioning with respect to $X_1$ gives the result for the process restricted to odd steps. Combining these yields the assertion.
\end{proof}

With an additional assumption that $\alpha$ is Diophantine and $\p$ is sufficiently smooth we  obtain the polynomial rate of mixing for \eqref{walk_def}. An important step in the proof is solving the Poisson equation (Theorem \ref{theorem_poisson}), which also implies the central limit theorem even without estimating the rate of mixing (Corollary~\ref{CrCLTEnv}). Then using Theorem \ref{theorem_poisson} we can modify the proof of Theorem \ref{theorem_mixing} to get the rate of convergence (Theorem \ref{theorem_rate}).
Denote
$$
\|\psi\|_{C^r} := \max\big\{ \|\psi\|_\infty, \| \psi' \|_\infty,\cdots, \|\psi^{(r)}\|_\infty \big\}.
$$

\begin{theorem}\label{theorem_poisson}
Let $\alpha\not\in\mathbb{Q}$ be Diophantine of type $(c,\tau)$, $m_0$ the lowest integer with $m_0>1+\tau$. Let $r\ge 0$, $\mathfrak{p}\in C^{r+3m_0}(\mathbb{T})$ symmetric, and let $\mu$ be the unique stationary measure for (\ref{walk_def}). Then there exists a constant $A>0$ such that for every $\psi\in C^{r+2m_0}(\mathbb{T})$ with $\mu(\psi)=0$ 
 the Poisson equation 
$$T\varphi - \varphi=\psi$$
admits a solution $\varphi$ which is
$C^{r}$ and $\|\varphi\|_{C^r}\le A\| \psi \|_{C^{r+2m_0}}$. 

If $\p$ is asymmetric we have a stronger estimate. Namely, it suffices to assume that 
$\p \in C^{r+2m_0}$, $\psi \in C^{r+m_0}$. In that case $\|\varphi\|_{C^r}\le A\| \psi \|_{C^{r+m_0}}$ for some constant $A$ independent of $\psi$.
\end{theorem}
\begin{corollary}
\label{CrCLTEnv}
If $\alpha$, $\p$ and $\psi$ are  as above, $(X_n)$ is the stationary Markov process with transition operator (\ref{walk_def}), 
 then the process $ \DS  \left(\sum_{n=1}^N \psi(X_n)\right)$ satisfies the functional Central Limit Theorem.
\end{corollary}
\begin{proof}[Proof of the Corollary]
We can apply the martingale decomposition method. Under the assumptions of Theorem \ref{theorem_poisson} the Poisson equation possesses a continuous solution $\varphi$.
$$
\sum_{n=1}^N \psi(X_n)
= \sum_{n=1}^N \bigg( T\varphi(X_n) - \varphi(X_n) \bigg) = \sum_{n=1}^{N-1} \bigg( T\varphi(X_n) - \varphi(X_{n+1}) \bigg) + T\varphi(X_N) - \varphi(X_1).
$$
The last two terms become negligible when divided by $\sqrt{N}$, and the sum is a square integrable martingale with stationary ergodic increments. Theorem 3 in \cite{Brown_71} completes the proof. 
\end{proof}

\begin{theorem}\label{theorem_rate}
Let $\alpha\not\in\mathbb{Q}$ be Diophantine of type
 $(c,\tau)$, $m_0$ the lowest integer with $m_0>\tau+1$. Let $k\ge 1$ and $r = 6km_0$. If $\p\in C^{r+m_0}(\T)$ is symmetric, then there exists a constant $A>0$ such that for every $x\in \T$ and $\psi \in C^r(\T)$
 $$
 \bigg| \EXP_x \psi (X_n) - \nu(\psi) \bigg| \le A \| \psi \|_{C^r}  n^{-k/2} \ln n,
 $$
 where $(X_n)$ is the Markov process (\ref{walk_def}), $\nu$ is the unique stationary distribution. If $\p\in C^{r+m_0}$ is asymmetric then the same is true with $r=4km_0$.
\end{theorem}

In Theorems \ref{theorem_poisson} and \ref{theorem_rate} the assumptions that $\alpha$ is Diophantine and $\psi$ is sufficiently smooth are both unremovable. Indeed, it has been shown (see Theorem 3 in \cite{Czudek_22}) that for every $\alpha$ Lioville it is possible to construct a $C^\infty$ observable for which the Central Limit Theorem fails. In view of Corollary \ref{CrCLTEnv} Theorem \ref{theorem_poisson} cannot hold for such observable. In the same way Theorem 2 in \cite{Czudek_22} implies that the observable must be sufficiently smooth even if $\alpha$ is Diophantine.

In order to give a counterexample to Theorem \ref{theorem_rate} more work needs to be done. 

\begin{theorem}\label{theorem_liouvillle}
If $\alpha$ is Liouville, $\p\in C^\infty(\T)$ with $\eps\leq \p\leq 1-\eps$,
then there exist $\varphi \in C^\infty(\T)$ and $\cG\subseteq \T$ of positive Lebesgue measure such that for every $x\in \cG$ and $\beta>0$ there exist infinitely many $N$'s with
\begin{equation}
\label{LioTimes}
\left| \EXP_x \varphi(X_N) - \int_\T \varphi(z) d \nu(z) \right|> \frac{1}{N^\beta},
\end{equation}
where $(X_n)$ evolves with the rule (\ref{walk_def}) and $\nu$ is the unique stationary measure.
\end{theorem}

\begin{remark}
It is known (see \cite{Conze_Guivarch_00}) that $\nu$ is equivalent to Lebesgue in the asymmetric
case. However, in the symmetric case the equivalence holds only 
if $\ln \p-\ln \q$ is a coboundary. This condition fails for generic pair $(\p, \alpha)$, see
\cite{DFS_21}. 
If $\nu$ is singular with respect to the Lebesgue measure then 
\eqref{LioTimes} does not rule out that $\nu$ is polynomially mixing in the sense
that 
$\DS \int \psi(x) \EXP_x(\varphi(X_n))d\nu(x)-\int \psi(x) d\nu(x) \int \phi(x) d\nu(x) $
decays polynomially. Thus the rate of mixing in the generic symmetric Liouville 
environment remains an open question. On the other hand, Theorem \ref{theorem_liouvillle} shows that
the assumptions that $\alpha$ is Diophantine in Theorem \ref{theorem_rate} is necessary.

We note that a related results in the case of constant $\p$ are obtained
in \cite[\S 4.3]{D02}.

\end{remark}

\subsection*{Acknowledgements}
The authors are grateful to Agnieszka Zelerowicz and the anonymous referee for useful comments. The first author is grateful to Minsung Kim for providing some references.
The first author has been supported by the European Union Horizon 2020 research and innovation programme under the Marie Sklodowska-Curie Grant Agreement No. 101034413.
The second author has been supported by the NSF grant DMS 2246983.

\section{The proof of Theorem \ref{theorem_mixing}}



\noindent
We follow closely  the strategy of \cite[Section 9]{Dolgopyat_Goldsheid_21}, see also \cite[Section 6]{DDKN}. 
We shall use

\begin{proposition}
\label{PrErg}
If $\{X_n\}$ has unique stationary measure $\nu$ then $\forall \phi\in C(\cX)$, \;

$\DS \frac{1}{N}\sum_{n=1}^N \EXP_x(\phi(X_n))=\nu(\phi)$ as $N\to\infty$ uniformly in $x.$
\end{proposition}

Given $X$ we consider the accelerated walk $\tX$ obtained from $X$ by erasing all repetitions 
(i.e. the points $X_n$ s.t. $X_n=X_{n-1}$). Given a segment $W$ of the accelerated walk
(i.e. a finite sequence $x_0, x_1, \dots, x_k$ such that $x_j\in (\Gamma\setminus \id) x_{j-1}$)
we let $t_W$ be the time it takes the walker to traverse $W$ (given that she takes that path). 
Thus $\DS t_W=\sum_{w\in W} \ell_w$, where $(\ell_w)_{w\in W}$ are independent and each $\ell_w$, $w\in W$, has geometric distribution  with parameter
$1-p(w, \id).$ Let $\DS T_W=\EXP_W(t_W)=\sum_{w\in W} \frac{1}{1-p(w, \id)}.$
 Let $\tau$ be the first time when $T_{\tX(1, \tau)}\geq \varepsilon_0 n/2$, where $\tX(1, k)$ stands for $(\tX_i)_{0\le i\le k}$. Thus the accelerated walk 
$\tX(1, \tau)$ belongs to the set 
$\fS_n(x)$ of accelerated walk segments starting at $x$ and such that $T_W\geq \varepsilon_0 n/2$ but 
$T_{\bar W}<\varepsilon_0 n/2$ for each prefix $\bar W\subset W.$ 
By Proposition \ref{moderate_deviations}, we have $\Prob_W(t_W>n)=O(\exp(-c(\log n)^2))$ for some $c>0$ and any segment $W$.
Therefore
$$ \EXP_x(\phi(X_n))
= \sum_{W\in \fS_n} \EXP_x\big((\phi(X_n)\mathds{1}\{ \tX(1, \tau) = W \}\big)
$$
$$=\sum_{W\in \fS_n} 
\sum_{k=1}^n \Prob_x (\tX(1, \tau)=W) 
\Prob_W(t_W=k) \EXP_{e(W)}(\phi(X_{n-k})) + O(\exp(-c(\log n)^2)),$$
where $e(W)$ is the endpoint of $W$.

We claim that $\forall \phi\in C(\cX)$ $\forall \eps>0$ $\exists n_0$ such that $\forall n\geq n_0$
$\forall W\in \fS_n$ we have

\begin{equation}
\label{FixedStart}
\left|\sum_{k=1}^n \Prob_x (\tX(1, \tau)=W) 
\left[\Prob_W(t_W=k) \EXP_{e(W)}(\phi(X_{n-k}))-\nu(\phi)\right]\right|\leq \eps. 
\end{equation}
Summing \eqref{FixedStart} over $W\in \fS_n$ we obtain the theorem. It remains to prove
\eqref{FixedStart}.

We use the following result. Let $\DS \sigma_W^2=\sum_{w\in W} \frac{p(w, \id)}{[1-p(w, \id)]^2}$ be the 
variance of $t_W.$ Since we assumed there exists $\eps_0$ such that $\eps_0 \leq p(x, \id) \leq 1-\eps_0$ for every $x\in \cX$, there exist $c_1<c_2$ such that $c_1 n\leq \sigma_W^2\leq c_2 n$
for all $W\in \fS_n.$

\begin{proposition}
\label{PrLLT}
(\cite{D-MD}).\;
We have $\DS \Prob_W\left(t_W=\varepsilon_0 n/2+j\right)=\frac{1}{\sqrt{2\pi} \; \sigma_W} e^{-j^2/(2\sigma_W^2)}+
o\left(\sigma_W^{-1}\right),$ where $o\left(\sigma_W^{-1}\right)$ decays uniformly in $j\in \mathbb{Z}$.
\end{proposition}

Divide $\mathbb{Z}$ into intervals $\{I_s\}$ of length $\delta \sqrt{n}$ for small $\delta$
then 
$$
 \sum_{\{j: \varepsilon_0n/2+j\in [1,n] \}} \Prob_W(t_W=\varepsilon_0 n/2+j) \EXP_{e(W)} (\phi(X_{n-(\varepsilon_0 n/2+j)}))
 $$$$
 =
\left[\sum_{\{s: I_s\subseteq [1,n]\}} \Prob_W(t_W\in I_s) \frac{1}{|I_s|} \sum_{m\in I_s} \EXP_{e(W)} (\phi(X_{n-m}))\right]+o_{\delta\to 0}(1)+o_{n\to\infty}(1)
$$
\begin{equation}\label{basic_eq}
=\left[\sum_s \Prob_W(t_W\in I_s) \left(\nu(\phi)+o_{n\to\infty}(1)\right)\right]+o_{\delta\to 0}(1)
=\nu(\phi)+o_{n\to\infty}(1)+o_{\delta\to 0}(1)
\end{equation}
where the first equality follows by Proposition \ref{PrLLT} and the second equality follows from 
Proposition \ref{PrErg}. Thus \eqref{FixedStart} follows and the theorem is proven.


\section{The proof of Theorem \ref{theorem_poisson}}

In the course of the proof we shall need the following lemmata.
\begin{lemma}\label{lemma_product_norm}
  $\| f g\|_{C^r} \le 2^r \|f\|_{C^r} \|g \|_{C^r}$ for every $r\ge 1$ and $f, g \in C^r$.
\end{lemma}
\begin{proof}
For $r=0$ the statement is obvious. For $r>1$
$$\| fg \|_{C^r}
= \max \big\{ \|(fg)'\|_{C^{r-1}}, \|f g\|_\infty \big\}
$$
$$
\le \max \big\{ 2^{r-1} \big( \|f'\|_{C^{r-1}} \|g\|_{C^{r-1}}+ \| f\|_{C^{r-1}} \|g'\|_{C^{r-1}} \big), \|f\|_\infty \|g\|_\infty \big\}.
$$
The assertion follows since all the norms $\|f\|_\infty$, $\|f\|_{C^{r-1}}$, $\|f'\|_{C^{r}}$ are bounded by $\|f\|_{C^{r}}$ and the same is true for $g$.
\end{proof}

\begin{lemma}\label{lemma_cohomological}
Let $\alpha$ be Diophantine of type $(c,\tau)$, $m_0$ the lowest integer with $m_0>\tau+1$. For every $r\ge 0$ there exists a constant $A_r$ such that for every $\psi\in C^{r+m_0}(\T)$ with $\int_{\T} \psi(x)dx=0$ the solution $\varphi$ of the cohomological equation $\varphi(x+\alpha)-\varphi(x)=\psi(x)$, $x\in\T$, is $C^{r}$ and $\|\varphi\|_{C^r}\le A_r \| \psi \|_{C^{r+m_0}}$.
\end{lemma}
\begin{proof}
It suffices to write down the formal solution $\varphi$ in terms of Fourier series and use the estimates on the $C^r$ norms, see e.g \cite{Llave}, the top of p.26. 
\end{proof}

\begin{proof}

\noindent\textbf{I. Symmetric case}

 By Lemma \ref{lemma_cohomological} under the above assumptions the cohomological equation
$$\frac{\mathfrak{p}(x)}{\mathfrak{q}(x)}=\frac{g(x+\alpha)}{g(x)}$$
possesses a positive solution $g\in C^{r+2m_0}(\mathbb{T})$. Then $\frac{g(x)}{\mathfrak{q}(x)}$ is the unique invariant density up to a multiplication by a constant (see \cite{Sinai_99}). Since $\nu(\psi)=0$ we have
$$\int_\mathbb{T} \psi(x)\frac{g(x)}{\mathfrak{q}(x)}dx=0,$$
and therefore Lemma \ref{lemma_cohomological} implies the equation
$f(x+\alpha)-f(x) = g(x)\psi(x)/\q(x)$ has a $C^{r+m_0}$ solution $f$. Since $g>0$, this solution can be written as $f(x)=g(x)\eta(x)$ for some $\eta\in C^{r+m_0}(\T)$. Thus $\eta$ satisfies
$$g(x+\alpha)\eta(x+\alpha) - g(x)\eta(x) = \frac{g(x)\psi(x)}{\q(x)}.$$ 
Since for any $\eta$ and any $c\in\mathbb{R}$ the function $\eta(x)+\frac{c}{g(x)}$ is a $C^{r+m_0}$ solution of the above equation as well, $\eta$ can be chosen so that $\int_{\mathbb{T}} \eta(x)dx=0$. By Lemmata \ref{lemma_cohomological} and \ref{lemma_product_norm}
$$
\|\eta\|_{C^{r+m_0}}
 = \bigg\|g\eta\cdot \frac{1}{g}\bigg \|_{C^{r+m_0}} 
\le 2^{r+m_0} \bigg\|\frac{1}{g} \bigg\|_{C^{r+m_0}} \cdot \|g\eta \|_{C^{r+m_0}}
$$
$$
\le  2^{r+m_0} \bigg\|\frac{1}{g}\bigg \|_{C^{r+m_0}} A_{r+m_0} \bigg\| \frac{g\psi}{\q}   \bigg\|_{C^{r+2m_0}}
\le  2^{2r+3m_0} A_{r+m_0} \bigg\|\frac{1}{g}\bigg \|_{C^{r+m_0}}  \bigg\| \frac{g}{\q}   \bigg\|_{C^{r+2m_0}} \|\psi\|_{C^{r+2m_0}} .
$$

Now, let $\varphi\in C^{r}$ be a solution of
$$\varphi(x) - \varphi(x-\alpha) = \eta(x).$$
By Lemma \ref{lemma_cohomological} and the above estimates $\|\varphi\|_{C^r} \le A \|\psi \|_{C^{r+2m_0}}$ for some constant $A$ independent of the choice of $\psi\in C^{r+2m_0}$.

We claim that $\varphi$ solves the Poisson equation. Indeed,
$$T\varphi(x)-\varphi(x)=\mathfrak{p}(x)\varphi(x+\alpha) + \mathfrak{q}(x)\varphi(x-\alpha) - \varphi(x) 
$$
$$
= \mathfrak{p}(x) \big[\varphi(x+\alpha) - \varphi(x) \big] - \mathfrak{q}(x)\big[\varphi(x)-\varphi(x-\alpha)\big]$$
$$
=\mathfrak{q}(x) \bigg[ \frac{\mathfrak{p}(x)}{\mathfrak{q}(x)} \big[ \varphi(x+\alpha) - \varphi(x) \big] - \big[ \varphi(x)-\varphi(x-\alpha) \big]$$
$$= \mathfrak{q}(x) \bigg( \frac{g(x+\alpha)}{g(x)}\eta(x+\alpha) - \eta(x) \bigg)$$
$$= \frac{\mathfrak{q}(x)}{g(x)} \bigg( g(x+\alpha)\eta(x+\alpha) - g(x)\eta(x) \bigg)=\psi(x).$$
\\

\noindent\textbf{II. Asymmetric case}
Let $\lambda=\exp\int_{\mathbb{T}}\log\frac{\mathfrak{p}(x)}{\mathfrak{q}(x)}dx$. Let us the assume $\lambda>1$. In the asymmetric case the invariant density (see \cite{Sinai_99}) is $\frac{\eta(x)g(x)}{\mathfrak{p}(x)}$, where $g$ solves
$$\frac{\mathfrak{p}(x)}{\mathfrak{q}(x)} = \lambda \frac{g(x)}{g(x-\alpha)},$$
and $\eta$ solves
$$\lambda^{-1}\eta(x+\alpha)-\eta(x)=\frac{1}{g(x)}.$$
Observe both $g$, $\eta$ are $C^{r+m_0}$ and $g$ is positive.

The function $\DS \kappa(x)=\sum_{k=0}^\infty \frac{g(x-k\alpha) \psi(x-k\alpha)}{\p(x-k\alpha)}\lambda^{-k}$ solves the equation
\begin{equation}
\lambda \kappa(x)-\kappa(x-\alpha)= \frac{\lambda g(x)\psi(x)}{\mathfrak{p}(x)}.
\end{equation} 
It is clear that $\kappa$ is $C^{r+m_0}$ and $\| \kappa \|_{C^{r+m_0}} \le \widetilde{A} \| \psi \|_{C^{r+m_0}}$ for some constant $\widetilde{A}$ independent of $\psi\in C^{r+m_0}$. Observe $\int_{\mathbb{T}} \frac{\kappa(x)}{g(x)}dx=0$. Indeed, by the definition of $\mr$ we have
$$\int_\mathbb{T} \frac{\kappa(x)}{g(x)}dx = \int_\mathbb{T} \kappa(x)\mr(x+\alpha)\lambda^{-1}dx - \int_\mathbb{T} \kappa(x)\mr(x)dx $$
$$
=\sum_{k=0}^\infty  \int_\mathbb{T}\frac{g(x-k\alpha) \psi(x-k\alpha)}{\p(x-k\alpha)}\lambda^{-(k+1)}\mr(x+\alpha)dx 
- \sum_{k=0}^\infty \int_\mathbb{T} \frac{g(x-k\alpha) \psi(x-k\alpha)}{\p(x-k\alpha)}\lambda^{-k}\mr(x)dx.
$$
Since the Lebesgue measure is rotation invariant we get 
$$\int_\mathbb{T} \frac{\kappa(x)}{g(x)}dx= \sum_{k=1}^\infty \int_\mathbb{T} \frac{g(x-k\alpha) \psi(x-k\alpha)}{\p(x-k\alpha)}\lambda^{-k}\mr(x)dx
-\sum_{k=0}^\infty \int_\mathbb{T} \frac{g(x-k\alpha) \psi(x-k\alpha)}{\p(x-k\alpha)}\lambda^{-k}\mr(x)dx
$$
$$= -\int_\mathbb{T} \frac{g(x)\mr(x)}{\p(x)}\psi(x)dx=0,$$
where the last step follows from the assumption that $\psi$ is centered. Since $\int_\mathbb{T} \frac{\kappa(x)}{g(x)}dx=0$ there exists $\varphi\in C^{r}(\mathbb{T})$ with 
$$\varphi(x+\alpha)-\varphi(x)=\kappa(x)/g(x).$$
For the same reasons as in the symmetric case and the estimates on the norm of $\kappa$, $\|\varphi\|_{C^r} \le A \| \psi \|_{C^{r+m_0}}$ for some $A$ independent of $\psi\in C^{r+m_0}$.

We claim that $\varphi$ solves the Poisson equation. Indeed,
$$T\varphi(x)-\varphi(x)= \q(x)\bigg[ \frac{\p(x)}{\q(x)}(\varphi(x+\alpha) -\varphi(x)) - (\varphi(x)- \varphi(x-\alpha)) \bigg]$$
$\DS
=\frac{\q(x)}{g(x-\alpha)} \bigg[ \lambda g(x) \frac{\kappa(x)}{g(x)} - g(x-\alpha) \frac{\kappa(x-\alpha)}{g(x-\alpha)} \bigg]
=\frac{\q(x)}{g(x-\alpha)} \big( \lambda \kappa(x) - \kappa(x-\alpha)\big) = \psi(x).
$
\end{proof}

\section{The proof of Theorem \ref{theorem_rate}}

The proof below is for the symmetric case. The asymmetric case is an obvious modification.

Fix $x\in \T$ and let $\widetilde{T} = T^2$ be the transition operator of (\ref{walk_def}) restricted to even times. It is clear from the proof of Theorem \ref{theorem_mixing} applied to the process associated to $\widetilde{T}$ that the assertion follows if the left-hand side of \eqref{FixedStart} with $T$ replaced by $\widetilde{T}$ decays faster than $A \| \psi \|_{C^r}  n^{-k/2} \ln n$ for some constant $A$ independent of $\psi$. Recall that $W\in\mathbb{S}_n(x)$ is a segment of an accelerated walk and $t_W$ is the time it takes the walker to traverse $W$. As explained before $t_W$ is a sum of random variables with geometric distributions and parameters uniformly separated from $0$ and $1$. Fix $n$ large and define inductively $\delta_j^0= \Prob( t_W = n-j)$, $j=0,1,\ldots, n-1$ and $\delta_j^m = \delta_{j-1}^{m-1} - \delta_{j} ^{m-1}$, $j=m, m+1, \cdots, n-1$, $m=1,\ldots, n-2 $.

Let $k=1$ and $r = 6km_0$. By Theorem \ref{theorem_poisson} there exists $\widetilde{A}$ such that for each $\psi\in C^{r}(\T)$ there exists $\varphi \in C^{r-6m_0}(\T)$ with
\begin{equation}\label{poisson_1}
\widetilde{T}\varphi - \varphi = T^2 \varphi - \varphi = \psi
\end{equation}
and $\| \varphi \|_{C^{r-6m_0}} \le \widetilde{A} \| \psi \|_{C^r}$. Indeed, if $\hat{\phi}$ solves $T\hat{\phi} - \hat{\phi}=\psi$ then the solution of $T \varphi - \varphi = \hat{\phi}$ solves also (\ref{poisson_1}). The solution $\varphi$ clearly satisfies $\wT^{j}\psi = \wT^{j+1} \varphi - \wT^j \varphi$ for every $j \ge 0$. 
Using Proposition \ref{moderate_deviations} from the appendix
and the bound $\| \wT^j \varphi \|_\infty \le \|\varphi \|_{C^{r-6m_0}} \le \widetilde{A} \|\psi\|_{C^r}$ 
we can estimate the left-hand side of \eqref{FixedStart} as follows:
$$
\sum_{j=0}^{n-1} \Prob_W (t_W= n-j) \wT^j \psi (e(W))
= \sum_{j=0}^{n-1} \Prob_W (t_W= n-j) \big( \wT^{j+1} \varphi (e(W)) - \wT^j \varphi (e(W)) \big)
$$
\begin{equation}\label{abel_summation1}
= \wT^n \varphi(e(W)) \Prob_W(t_W=1) - \varphi(e(W)) \Prob_W (t_W= n)
+\sum_{j=1}^{n-1} \wT ^j \varphi (e(W)) \delta_j^1
\end{equation}
$$
=\sum_{\{j: |n-j-\varepsilon_0n/2|<\sqrt{n}\ln n\} } \wT ^j \varphi (e(W)) \delta_j^1+ \| \psi \|_{C^r} O(\exp(-c(\ln n)^2)).
$$
The second term decays faster than polynomially. 
By Proposition \ref{decay_delta} from the appendix, the first term is bounded by
$$
\bigg| \sum_{\{j: |n-j-\varepsilon_0n/2|<\sqrt{n}\ln n\} }  \wT ^j \varphi (e(W)) \delta_j^1 \bigg| \le 2\widetilde{A}\sqrt{n} \ln n \| \psi \|_{C^{r}} \max_{j} \delta_j^1  = O(n^{-1/2} \ln n) \| \psi \|_{C^{r}},
$$
which gives the assertion for $k=1$.

To show the assertion for $k=2$ we again use Theorem \ref{theorem_poisson} to find a function $\widetilde{\varphi}$ such that $\wT \widetilde{\varphi} - \widetilde{\varphi} = \varphi$ and $\| \widetilde{\varphi} \|_{C^{r-12m_0}} \le \widetilde{A}^2 \| \psi \|_{C^r}$ (it is possible since $r = 12m_0$ by the assumption). Then the second line of (\ref{abel_summation1}) (and thus \eqref{FixedStart}) can be rewritten as
$$
 \wT^n \varphi(e(W)) \Prob_W(t_W=1) - \varphi(e(W)) \Prob_W (t_W= n)
+\sum_{j=1}^{n-1} \wT ^j \varphi (e(W)) \delta_j^1
$$
$$
=  \wT^n \varphi(e(W)) \Prob_W(t_W=1) - \varphi(e(W)) \Prob_W (t_W= n)
+\sum_{j=1}^{n-1} \big( \wT ^{j+1} \widetilde{\varphi} (e(W)) - \wT^j \widetilde{\varphi}(e(W)) \big) \delta_j^1
$$
$$
= \wT^n \varphi(e(W)) \Prob_W(t_W=1) - \varphi(e(W)) \Prob_W (t_W= n)
$$
$$
+ \wT^n \widetilde{\varphi}(e(W)) \delta_{n-1}^1 - \wT \widetilde{\varphi}(e(W)) \delta_1^1 +\sum_{j=2}^{n-1} \wT^j\widetilde{\varphi} (e(W))  \delta_j^2.
$$
By exactly the same argument like for $k=1$ the above expression equals
$$
\sum_{\{j: |n-j-\varepsilon_0n/2|<\sqrt{n}\ln n\} } \wT ^j \varphi (e(W)) \delta_j^2+ \| \psi \|_{C^r} O(\exp(-c(\ln n)^2))),
$$
where the first term  
can be bounded by Proposition \ref{decay_delta}  as
$$
\bigg| \sum_{\{j: |n-j-\varepsilon_0n/2|<\sqrt{n}\ln n\} }  \wT ^j \varphi (e(W)) \delta_j^2 \bigg| \le 2\sqrt{n} \ln n \| \psi \|_{C^{r}} \max_{j} \delta_j^2  = O(n^{-1} \ln n) \| \psi \|_{C^{r}}.
$$
This completes the proof for $k=2$.

The claim for $k>2$ is obtained in exactly the same way by repeatedly
solving the Poisson equation and using higher order Abel summation.

\section{Slow mixing.}

Here we prove Theorem \ref{theorem_liouvillle}.

Denote by $G_q^+$ the set of points in $\T$ whose distance to the set 
$\DS \left\{0, \frac{1}{q}, \cdots, \frac{q-1}{q} \right\}$ is less than $\DS \frac{1}{16q}$
 and let $\DS G_q^-=G^+_q+\frac{1}{2q}.$
We shall need the following lemma.

\begin{lemma}\label{lemma_liouville}
If $\alpha\in \mathbb{R}$, $p, q\in \mathbb{Z}$, $\gamma\ge 2$ satisfy
$\DS
|q\alpha - p |\!\!<\!\! \frac{1}{16q^\gamma}
$
then 
$\EXP_x  \cos (2\pi q X_{\widetilde{q}})\!\!>\!\!\frac{\sqrt{2}}{2} $ for $x\in G_q^+$, where $\widetilde{q} = \lfloor q^{\gamma-1} \rfloor$.
 Likewise 
$\EXP_x  \cos (2\pi q X_{\widetilde{q}})\!\!<\!\!-\frac{\sqrt{2}}{2} $ for $x\in G_q^-$.
\end{lemma}
\begin{proof}
Fix $p, q, \gamma, \alpha$ as in the statement. Clearly $X_{\widetilde{q}}$ started at $x\in\T$ can attain with positive probability only those points $x+n\alpha$ with 
$|n|\le  \tq$. If $x\in G^+_q$, then the distance of $x+n\alpha$ with $|n|\le \widetilde{q}$ to the set $\{0, 1/q, \cdots, (q-1)/q \}$ is less than $\frac{1}{16q}+ q^{\gamma-1}\cdot \frac{1}{16q^\gamma}< \frac{1}{8q}$. Since $\cos(2\pi q x)$ is $1/q$ periodic we get $\cos(2 \pi q X_{\widetilde{q}}) > \cos(\pi / 4) = \frac{\sqrt{2}}{2}$ a.s., which implies the first assertion, the second is similar.
\end{proof}

\begin{proof}[Proof of Theorem \ref{theorem_liouvillle}]
We are going to construct inductively a sequence $(\varphi_n)$ of functions of the form $a_n\cos(2\pi q_n x)$ in such a way that the sum $\DS \varphi=\sum_{n=1}^\infty \varphi_n$ satisfies the assertion.

By the assumptions for any $\gamma\ge 2$ there exists infinitely many $p, q$'s such that
\begin{equation}\label{l1}
|q\alpha - p | < \frac{1}{16q^\gamma}.
\end{equation}
Let $p_1, q_1$ be an arbitrary pair satisfying \eqref{l1} with $\gamma=2$. 
Let  $\widetilde{q_1} = \lfloor q_1^{\gamma-1} \rfloor$, 
$\varphi_1(x) =  q_1^{-\sqrt{2}} \cos (2\pi q_1 x)$. By Lemma \ref{lemma_liouville} 
either 
$\nu(\varphi_1)\leq 0$ and 
$\DS \EXP_x \varphi_1(X_{\widetilde{q}_1})-\nu(\varphi_1) > \frac{\sqrt{2}}{2} q_1^{-\sqrt{2}} $ for every $x\in G_{q_1}^+$ 
or 
$\nu(\varphi_1)\geq 0$ and 
$\DS \left|\EXP_x \varphi_1(X_{\widetilde{q}_1})-\nu(\varphi_1)\right| > \frac{\sqrt{2}}{2} q_1^{-\sqrt{2}} $ for every $x\in G_{q_1}^-.$


Assume $\varphi_1, \cdots, \varphi_{n-1}$ are already defined, $n\ge 1$. Let $p_n, q_n$ be such that (\ref{l1}) is satisfied for $\gamma = n+1$, and put $\hphi_n(x) = q_n^{-\sqrt{n+1}}\cos(2\pi q_n x)$, $\widetilde{q}_n = \lfloor q_n^{n} \rfloor$. We impose additional condition that $q_n$ is so large that  $q_n^{-\sqrt{n+1}}<0.001 q_{n-1}^{-\sqrt{n}} .$

To fix the notation consider the case where $\nu(\hphi_n)\geq 0.$
Then for all $x\in G_{q_n}^-$ we have
$\DS \left|\EXP_x(\hphi_n(X_{\widetilde{q}_n}))-\nu(\hphi)\right|\geq 
\frac{\sqrt{2}}{2} q_n^{-\sqrt{n+1}}.$

Denote $\DS \psi_n=\sum_{j=1}^n \varphi_j(x).$ Let
$H_n=\{x\in G_{q_n}^-: \EXP_x(\psi_{n-1} (X_{\tq_n}))\geq 
\nu(\psi_{n-1})+\frac{\sqrt{2}}{4} q_n^{-\sqrt{n+1}} \}.$

If $\Leb(H_n)\geq \Leb(G_{q_n}^-)/2$ then we set $\varphi_n=0$ and 
$\cG_n=H_n.$ 
If $\Leb(H_n)< \Leb(G_{q_n}^-)/2$ then we set $\varphi_n=\hphi_n$ and 
$\cG_n=G_{q_n}^-\setminus H_n.$ 
In either case, by Lemma \ref{lemma_liouville} on
$\cG_n$ we have 
\begin{equation}
\label{XNotMix}
\left|\EXP_x(\psi_{n} (X_{\tq_n}))-
\nu(\psi_{n})\right|\geq \frac{\sqrt{2}}{4} q_n^{-\sqrt{n+1}} 
=\frac{\sqrt{2}}{4} \tq_n^{\;\;-\frac{\sqrt{n+1}}{n}}.
\end{equation}

If $\nu(\hphi_n)<0$ then we proceed as above but will use $G_{q_n}^+$ instead of
$G_{q_n}^-$ to ensure  \eqref{XNotMix}.


Put $\DS \varphi=\sum_{n=1}^\infty \varphi_n$. Since  
$\DS \sum_{n=1}^\infty \varphi_n^{(j)}$ converges uniformly for each $j$ the function $\varphi$ is $C^\infty$. 

 Let $\cG=\limsup \cG_{n}$. Since  the Lebesgue 
measure of $\cG_{n}$ is uniformly bounded
from below, $\cG$ is necessarily of positive measure. Fix $x\in \cG$. Then for infinitely many $n$'s we have
$$
\left| \EXP_x \varphi(X_{\widetilde{q}_n})-\nu(\varphi)\right|
\ge \left| \EXP_x \psi_n(X_{\widetilde{q}_n})-\nu(\psi_n)\right|
-
\left|\sum_{j>n}
\EXP_x(\varphi_j(X_{\tq_n})-
\nu(\varphi_j) \right|  
$$
The last term is at most $0.003 q_n^{-\sqrt{n+1}}$ 
due to our choice of $n$
while the first term is at least $\frac{\sqrt{2}}{4} q_n^{-\sqrt{n+1}}$ 
by \eqref{XNotMix}. It follows that for infinitely many $n$
$$ \left| \EXP_x \varphi(X_{\widetilde{q}_n})-\nu(\varphi)\right|\geq 
0.3 q_n^{-\sqrt{n}}=0.3 \tq_n^{\;\;-\frac{\sqrt{n+1}}{n}}. $$
Since $\frac{\sqrt{n+1}}{n}$ tends to 0 as $n\to \infty$ and the above inequality holds for infinitely many $n$'s by the definition of $\cG$, 
$ \EXP_x \varphi(X_{n})-\nu(\varphi)$ decays slower than polynomially.
\end{proof}

\appendix
\section{Sums of geometric random variables.}

\begin{proposition}\label{decay_delta}
Let $(\ell_j)_{j\ge 1}$ be a sequence of independent random variables with geometric distributions with parameters $p_j$. Let us assume there exists $\varepsilon_0>0$ such that $\varepsilon_0< p_j < 1-\varepsilon_0$, $j\ge 1$. Denote $S_n= \ell_1+\cdots+\ell_n$ and define 
$\delta_{0, n}(j)=\Prob(S_n=j),$ $\delta_{m,n}=\nabla^m \delta_{0, n}$
where the operator $\nabla$ acts on sequences by
$(\nabla a)(j)=a(j+1)-a(j).$ Then
$\DS \sup_j |\delta_{m,n}(j)|\leq C_m n^{-(m+1)/2}. $
\end{proposition}

\begin{proof}
Let $\phi_k(t)$ denote the characteristic function of $\ell_k\!\!-\!\!\EXP(\ell_k)$
and $\DS \Phi_n(t)\!\!=\!\!\prod_{k=1}^n \phi_k(t)$ be the characteristic function
of $S_n-\EXP(S_n).$ Then
$\DS \Prob(S_n\!\!=\!\!j)\!\!=\!\!\frac{1}{2\pi} \int_{-\pi}^{\pi} \Phi_n(t) e^{it(j+\EXP(S_n))} dt. $
So 
$$ 
 \delta_{m,n} (j)=\frac{1}{2\pi} \int_{-\pi}^{\pi} \Phi_n(t) e^{it(j+\EXP(S_n))} 
\left(e^{it}-1\right)^m dt .$$
Therefore 
\begin{equation}
\label{DeltamInt}
 \sup_j |\delta_{m,n} (j)|=\frac{1}{2\pi} 
 \int_{-\pi}^{\pi} |\Phi_n(t)| 
\left|e^{it}-1\right|^m dt 
\leq C \int_{-\pi}^{\pi} |\Phi_n(t)| |t|^m dt.
\end{equation}
We claim that given $\varepsilon_0$ as in the assumption of the proposition there are constants
$\delta, \kappa>0,$ and $\theta<1$ such that
\begin{equation}
\label{CharGeom1}
|\phi_k(t)|\leq e^{-\kappa t^2} \text{ for } |t|\leq \delta,
\end{equation}
and
\begin{equation}
\label{CharGeom2}
|\phi_k(t)|\leq \theta \text{ for }  \delta\leq |t|\leq \pi. 
\end{equation}
To check \eqref{CharGeom1}, 
note that  the Taylor series of $\phi_k(t)$ takes form
$$ \phi_k(t)=1-\frac{\mathrm{Var}(\ell_k^2 t^2)}{2}+\eps_k (t)t^3\quad\text{where}\quad |\eps_k(t)|\leq C \EXP(\ell_k^3) $$
for some constant $C$ independent of $k$ since under the assumptions of the proposition $\Var(\ell_k)$ is uniformly bounded from 
below while $\EXP(\ell_k^3)$ is uniformly bounded from above. 
Hence  if $\delta$ is sufficiently small, then for $|t|\!\!\leq\!\! \delta$ we have that $|\phi_k(t)|\!\!\leq\!\! 1-\kappa t^2\!\!\leq\!\! e^{-\kappa t^2}$ where 
$\DS \kappa=\inf_k \frac{\Var(\ell_k)}{4}$ and the last inequality relies on the fact that 
$1-s\leq e^{-s}$ for $s\geq 0.$

\eqref{CharGeom2} holds because for $\delta\leq |t|\leq \pi$
$$ |\phi_k(t)|=\left|\frac{p_k}{1-q_k e^{it}}\right|=\frac{p_k}{\sqrt{p_k^2+2q_k(1-\cos t)}}\leq \frac{1}{\sqrt{1+2\eta (1-\cos\delta)}} .
$$
Multiplying the above estimates we obtain that
$\DS  |\Phi_n(t)|\leq \begin{cases} e^{-\kappa n t^2}, &  |t|<\delta\\
\theta^n, & \delta\leq |t|\leq \pi. \end{cases}. $
Plugging this into \eqref{DeltamInt} we obtain
$$ \sup_j |\delta_{m,n} (j)|
\leq C\left[\theta^n+ \int_{-\delta}^{\delta} e^{-\kappa n t^2} t^m dt\right] .$$
The second term in the RHS is smaller than 
$\DS C \int_{-\infty}^{\infty} e^{-\kappa n t^2} t^m dt= O\big( n^{-(m+1)/2} \big)  $
as claimed.
\end{proof} 

\begin{proposition}\label{moderate_deviations}
Let $(\ell_j)_{j\ge 1}$ be a sequence of independent random variables with geometric distributions with parameters $p_j$ Assume that there exists $\varepsilon_0>0$ such that $\varepsilon_0< p_j < 1-\varepsilon_0$, $j\ge 0$. Denote $S_n= \ell_1+\cdots+\ell_n$ and for each $n$ set
$$
\tau = \min \bigg \{ k \ge 1 : \sum_{j=1}^k \frac{1}{p_j}>n/2 \bigg\}.
$$
Then there exists a constant $c>0$ such that
$$
\Prob \big( \big| S_\tau - n/2 \big|> \sqrt{n}\ln n \big) =O(\exp(-c(\ln n)^2)).
$$
\end{proposition}

\begin{proof}
 Since the parameters $p_k$ are bounded away from 0 and 1, 
$$\frac{n}{2}<\tau<\frac{n}{2}+O(1).$$
 It therefore suffices to show that for each $N$
$$
\Prob \big( \big| S_N - E S_N \big|> \sqrt{N}\ln N \big) = O(\exp(-c(\ln N)^2)).
$$
To this end let $\psi_k(t)$ denote the moment generating function of $\ell_k-\EXP(\ell_k)$
and 
$\DS \Psi_n(t)=\prod_{k=1}^n \psi_k(t)$ be the moment generating function
of $S_n-\EXP(S_n).$ Then similarly to the proof of \eqref{CharGeom1}
we obtain that for $|t|<\delta$ where $\delta$ is sufficiently small,
we have 
$\psi_k(t)\leq e^{\kappa t^2}$  and
so $\Psi_N(t)\leq e^{N \kappa t^2}. $ Thus for any such 
 $0<t<\delta$
$$ \Prob \big( S_N - E S_N > \sqrt{N}\ln N \big) \leq 
\exp\left[\kappa N t^2-t\sqrt{N} \ln N\right]. $$
Choosing $t=\ln N/(2\kappa \sqrt{N})$ 
get that 
$\DS \Prob \big( S_N -E S_N > \sqrt{N}\ln N \big) =O(\exp(-c(\ln N)^2)).$
The estimate $\DS \Prob \big( S_N -E S_N <- \sqrt{N}\ln N \big) =O(\exp(-c(\ln N)^2))$
is similar, using \\
$t=-\ln N/(2\kappa \sqrt{N}).$ 
\end{proof}

\bibliographystyle{plain}
\bibliography{Bibliography}

\end{document}